\documentclass[smallextended,referee,envcountsect,]{svjour3}
\smartqed

\newtheorem{assumption}{Assumption}[section]
\usepackage{amsmath,amssymb,mathtools,bm}
\usepackage{hyperref,lmodern,xcolor,subcaption}
\usepackage{float,booktabs,url,cite}
\usepackage{graphicx,tabularx}
\usepackage{algorithm,algorithmic}

\begin{document}

\allowdisplaybreaks

\newcommand{\argmax}{\operatornamewithlimits{argmax}}
\newcommand{\argmin}{\operatornamewithlimits{argmin}}

\title{Mini-Batch Stochastic Krasnosel'ski\u\i-Mann Algorithm for Nonexpansive Fixed Point Problems}

\author{Hideaki Iiduka}

\institute{Hideaki Iiduka \at
             Meiji University \\
              1-1-1 Higashimita, Tama-ku, Kawasaki-shi, Kanagawa 214-8571 Japan\\
              iiduka@cs.meiji.ac.jp
}

\date{Received: date / Accepted: date}

\maketitle

\begin{abstract}
The Krasnosel'ski\u\i-Mann algorithm is a well-known method for finding fixed points of a nonexpansive mapping with strong theoretical guarantees. However, there are practical large-scale problems to which this algorithm cannot be applied. Here, to resolve the issue caused by the computational difficulty of the mapping, we define a computable mini-batch stochastic mapping, which is a unbiased estimator of the nonexpansive mapping, and implement it in the Krasnosel'ski\u\i-Mann algorithm. We show that the algorithm with increasing batch sizes converges almost surely to a fixed point of the nonexpansive mapping. We also perform a convergence rate analysis on the algorithm.
\end{abstract}


\section{Introduction}
\subsection{Nonexpansive fixed point problem}
Fixed point problems \cite{goebel1,goebel2,b-c,takahashi} constitute a fundamental class of mathematical problems with wide-ranging applications. In particular, the {\em nonexpansive fixed point problem} is generally one of finding the fixed points of a nonexpansive mapping, and it includes significant problems, such as convex feasibility problems, monotone variational inequality problems, and convex minimization problems.

Many iterative algorithms for solving nonexpansive fixed point problems have been developed (see, e.g., \cite{berinde} for fixed point approximation algorithms); one especially useful method is the Krasnosel'ski\u\i-Mann algorithm \cite{kra,mann}. Here, let $T \colon \mathbb{R}^d \to \mathbb{R}^d$ be nonexpansive (see Assumption \ref{assum:1}(A1) for the definition of nonexpansivity). When the current point is the $k$-th approximate point $\bm{x}_k$, the Krasnosel'ski\u\i-Mann algorithm updates its iterates in the direction: $\bm{d}_k = T (\bm{x}_k) - \bm{x}_k$ with a step size $\alpha_k \in (0,1)$, i.e., 
\begin{align}\label{Kras_Mann}
\bm{x}_{k+1} 
= \bm{x}_k + \alpha_k \bm{d}_k
= \bm{x}_k + \alpha_k (T (\bm{x}_k) - \bm{x}_k)
= (1 - \alpha_k) \bm{x}_k + \alpha_k T (\bm{x}_k).
\end{align}
In theory, convergence of the Krasnosel'ski\u\i-Mann algorithm \eqref{Kras_Mann} to a fixed point of $T$ is guaranteed if $\alpha_k$ satisfies $\sum_{k=0}^{+ \infty} \alpha_k (1 - \alpha_k) = + \infty$ \cite[Theorem 5.15]{b-c}, \cite[Corollaries 1-3]{GROETSCH1972369}. Moreover, the algorithm satisfies that $\| \bm{x}_K - T (\bm{x}_K) \| = O(1/\sqrt{\sum_{k=0}^{K-1} \alpha_k (1 - \alpha_k)})$ \cite[Theorem 1]{Cominetti:2014aa}. Hence, by using a constant step size $\alpha_k = \alpha \in (0,1)$, the Krasnosel'ski\u\i-Mann algorithm converges with an $O(1/\sqrt{K})$ convergence rate more quickly than when using a diminishing step size (see also Section \ref{sec:3.2.1}).

Meanwhile, in practice, the computation of the Krasnosel'ski\u\i-Mann algorithm \eqref{Kras_Mann} strongly depends on the computability of $T$. For example, let us consider a minimization problem that is to find a global minimizer of the sum of many convex functions, denoted by $f \coloneqq \sum_{i=1}^n f_i$. Algorithm \eqref{Kras_Mann} with 
\begin{align}\label{many_convex}
T (\bm{x})  
\coloneqq \bm{x} - \beta \nabla f (\bm{x})
= \bm{x} - \frac{\beta}{n} \sum_{i=1}^n \nabla f_i (\bm{x})
= \frac{1}{n} \sum_{i=1}^n \underbrace{\left( \bm{x} - \beta \nabla f_i (\bm{x})   \right)}_{T_i (\bm{x})}
\end{align}
converges to a global minimizer of $f$ under sufficient conditions of $\beta$ for the nonexpansivity of $T$ (see Example \ref{exp:1}(iii)). However, since practical convex minimization problems have large $n$ and $d$, the computation of $T$ is expensive. Accordingly, it is not necessarily true that the Krasnosel'ski\u\i-Mann algorithm \eqref{Kras_Mann} can be applied to practical problems. 

\subsection{Stochastic Krasnosel'ski\u\i-Mann algorithms}
Here, we suppose that $T$ consists of a finite number of $T_1, \cdots, T_n$ (see also $T$ in the above subsection). When $T$ cannot be used at each iteration, the following stochastic nonexpansive mapping can be defined for all $\bm{x} \in \mathbb{R}^d$:
\begin{align*}
T_{\bm{\xi}} (\bm{x}) \coloneqq
\begin{dcases}
T (\bm{x}) + \bm{\xi} &\text{ (noisy observation)} \\
\frac{1}{b} \sum_{i=1}^b T_{\xi_i} (\bm{x}) &\text{ (mini-batch estimation)},
\end{dcases}
\end{align*}
where, in the ``noisy observation" case, the noise $\bm{\xi}$ of $T(\bm{x})$ is a random variable such that the expectation of $\bm{\xi}$ is $0$, i.e., $\mathbb{E}_{\bm{\xi}}[\bm{\xi}] = \bm{0}$, in the ``mini-batch estimation" case, $b$ $(\leq n)$ is the batch size (the number of samples), $\bm{\xi} = (\xi_1, \cdots, \xi_b)^\top$ comprises $b$ independent and identically distributed (i.i.d.) variables, and $T_{\xi_i} (\bm{x})$ is a unbiased estimator of $T$, i.e., $\mathbb{E}_{\xi_i}[T_{\xi_i} (\bm{x})] = T(\bm{x})$. Since 
\begin{align*}
\mathbb{E}_{\bm{\xi}} \left[ T_{\bm{\xi}} (\bm{x}) \right]
=
\begin{dcases}
T (\bm{x}) + \mathbb{E}_{\bm{\xi}}[\bm{\xi}] = T (\bm{x}) &\text{(noisy observation)} \\
\mathbb{E}_{\xi_i} \left[ \frac{1}{b} \sum_{i=1}^b T_{\xi_i} (\bm{x}) \right] = T (\bm{x})  &\text{(mini-batch estimation)},
\end{dcases}
\end{align*}
we expect that upon replacing $T$ with $T_{\bm{\xi}}$, algorithm \eqref{Kras_Mann} will converge to a fixed point of $T$. In fact, \cite[Theorem 2.5]{doi:10.1137/22M1515550} and \cite[Corollary 2.7]{doi:10.1137/140971233} showed that the stochastic Krasnoselski\u\i-Mann algorithm in the noisy observation case defined for all $k \in \{0\} \cup \mathbb{N}$ by 
\begin{align*}
\bm{x}_{k+1} 
= (1 - \alpha_k) \bm{x}_k + \alpha_k T_{\bm{\xi}_k} (\bm{x}_k)
= (1 - \alpha_k) \bm{x}_k + \alpha_k ( T (\bm{x}_k) + \bm{\xi}_k)
\end{align*}
converges almost surely to a fixed point of $T$ under certain conditions.

Mini-batch estimation is often used in practice. For example, let us consider empirical risk minimization in machine learning. Let $\bm{x} \in \mathbb{R}^{d}$ be a parameter of a deep neural network, $S = \{(\bm{d}_1,\bm{y}_1), \ldots, (\bm{d}_n,\bm{y}_n)\}$ be the training set, where each data point $\bm{d}_i$ is associated with a label $\bm{y}_i$, and let $f_i (\cdot) \coloneqq f(\cdot;(\bm{d}_i,\bm{y}_i)) \colon \mathbb{R}^{d} \to \mathbb{R}$ be the loss function corresponding to the $i$-th labeled training data $(\bm{d}_i,\bm{y}_i)$. Empirical risk minimization minimizes the empirical risk defined for all $\bm{x} \in \mathbb{R}^{d}$ as $f (\bm{x}) \coloneqq (1/n) \sum_{i=1}^n f_i(\bm{x})$. In general, the number $n$ of training data and the number of dimensions $d$ of the deep neural network are very large, and we cannot use $\nabla f(\bm{x}) = (1/n) \sum_{i=1}^n \nabla f_i(\bm{x})$ directly to minimize $f$ (see also \eqref{many_convex}). Here, instead of $\nabla f (\bm{x})$, we can use the mini-batch stochastic gradient of $f$ defined by 
\begin{align*}
\nabla f_{\bm{\xi}} (\bm{x}) \coloneqq \frac{1}{b} \sum_{i=1}^b \nabla f_{\xi_i} (\bm{x}).
\end{align*}
Many optimizers using the mini-batch stochastic gradient have been presented to minimize $f$. For example, mini-batch stochastic gradient descent (SGD) defined by 
\begin{align*}
\bm{x}_{k+1} 
= \bm{x}_k - \eta \nabla f_{\bm{\xi}_k} (\bm{x}_k)
= \bm{x}_k - \frac{\eta}{b} \sum_{i=1}^b \nabla f_{\xi_{k,i}} (\bm{x}_k)
= \frac{1}{b} \sum_{i=1}^b \underbrace{\left( \bm{x}_k - \eta \nabla f_{\xi_{k,i}} (\bm{x}_k)   \right)}_{T_{\xi_{k,i}} (\bm{x}_k)}
\end{align*}
can minimize $f$ under certain conditions \cite[Section 5]{garrigos2024handbookconvergencetheoremsstochastic}, where $\eta > 0$ and $\xi_{k,i}$ is a random variable generated by the $i$-th sampling in the $k$-th iteration. The motivation behind this work is thus to show that the stochastic Krasnoselski\u\i-Mann algorithm converges in the mini-batch estimation case.

\subsection{Mini-batch stochastic Krasnosel'ski\u\i-Mann algorithm}
This paper considers algorithm \eqref{Kras_Mann} in the setting that replaces $T$ with the mini-batch stochastic mapping $T_{\bm{\xi}_k} \coloneqq (1/b_k) \sum_{i=1}^{b_k} T_{\xi_{k,i}}$, i.e., 
\begin{align}\label{mini_batch_KM}
\bm{x}_{k+1} 
= (1 - \alpha_k) \bm{x}_k + \alpha_k T_{\bm{\xi}_k} (\bm{x}_k)
= (1 - \alpha_k) \bm{x}_k + \frac{\alpha_k}{b_k} \sum_{i=1}^{b_k} T_{\xi_{k,i}} (\bm{x}_k),
\end{align}
where $b_k$ is the batch size at the $k$-th iteration. We show that, under certain conditions (see Assumption \ref{assum:1} and its Example \ref{exp:1}), algorithm \eqref{mini_batch_KM} with $(\alpha_k)$ and $(b_k)$ satisfying 
\begin{align}\label{conditions_1_1}
\sum_{t=0}^{+\infty} \alpha_k (1 - \alpha_k) = + \infty \text{ and }
\sum_{t=0}^{+ \infty} \frac{1}{\sqrt{b_k}} < + \infty
\end{align}
converges almost surely to a fixed point of $T$ (Theorem \ref{thm:1}). The condition on $b_k$ in \eqref{conditions_1_1} implies that increasing batch sizes are essential for guaranteeing almost-sure convergence of the algorithm \eqref{mini_batch_KM}. In sampling with replacement, even if the batch size $b_k$ exceeds $n$, $T_{\bm{\xi}_k} \neq T$ holds in general. Hence, to examine the convergence of mini-batch algorithms under sampling with replacement, we can use $b_k \to + \infty$ $(k \to + \infty)$. Intuitively, when the batch size (the number of samples) $b_k$ is large enough, $T_{\bm{\xi}_k} = (1/b_k) \sum_{i=1}^{b_k} T_{\xi_{k,i}}$ is approximately $T$. Hence, for large enough $k$, algorithm \eqref{mini_batch_KM} is approximately algorithm \eqref{Kras_Mann}. As a result, the convergence of algorithm \eqref{Kras_Mann} implies the convergence of the algorithm \eqref{mini_batch_KM}. Below, we theoretically prove almost-sure convergence of algorithm \eqref{mini_batch_KM} through Proposition \ref{prop:1} and Lemmas \ref{lem:1} and \ref{lem:2} (Subsection \ref{sec:3.1}). 
 
We also show that algorithm \eqref{mini_batch_KM} with \eqref{conditions_1_1} has the following convergence rate (Theorem \ref{thm:2}):
\begin{align*}
\min_{t \in [0:K-1]} \mathbb{E} \left[ \| \bm{x}_k - T (\bm{x}_{k})\| \right]
&= 
O \left( \frac{1}{\sqrt{\sum_{k=0}^{K-1} \alpha_k (1 - \alpha_k)}} \right),
\end{align*}
which, together with $\alpha_k = \alpha \in (0,1)$, implies that $\min_{t \in [0:K-1]} \mathbb{E} [ \| \bm{x}_k - T (\bm{x}_{k})\| ] = O(1/\sqrt{K})$, where $[0:K-1] \coloneqq \{0,1,\cdots, K-1\}$ for $K \in \mathbb{N}$. Condition \eqref{conditions_1_1} implies that algorithm \eqref{mini_batch_KM} with a constant batch size $b_k = b$ would not be applicable to nonexpansive fixed point problems. To emphasize our claim, we reconsider the empirical risk minimization in the above subsection. Theorems 3.1 and 3.2 in \cite{umeda2025increasing} show that mini-batch SGD with a constant step size $\alpha_k = \alpha$ and a constant batch size $b_k = b$ satisfies 
\begin{align*}
\min_{k \in [0:K-1]}\mathbb{E} \left[ \|\nabla f (\bm{x}_k)\| \right] 
= O \left( \sqrt{\frac{1}{K} + \frac{1}{b}}  \right)
\to \mathcal{O} \left(\frac{1}{\sqrt{b}} \right) \neq 0 \quad (K \to + \infty),
\end{align*}
while mini-batch SGD with a constant step size $\alpha_k = \alpha$ and an increasing batch size $b_k$ (e.g., $b_k$ doubly increases every epoch) satisfies 
\begin{align*}
\min_{k \in [0:K-1]}\mathbb{E} \left[ \|\nabla f (\bm{x}_k)\| \right] 
= O \left( \frac{1}{\sqrt{K}}  \right)
\to 0 \quad (K \to + \infty),
\end{align*}
where the notation $\mathcal{O}$ is used in distinction from the Landau symbol $O$ and $C = \mathcal{O}(1/\sqrt{b})$ implies there exists $M > 0$ such that $C \leq M/\sqrt{b} \neq 0$. Hence, the use of increasing batch sizes is needed to guarantee theoretical convergence of algorithm \eqref{mini_batch_KM}.

Nevertheless, when using an increasing batch size $b_k$, there exists $k_0 \in \mathbb{N}$ such that, for all $k \geq k_0$, $b_k \geq n$. This implies that, for all $k \geq k_0$, computing $T_{\bm{\xi}_k} (\bm{x}_k) = (1/b_k) \sum_{i=1}^{b_k} T_{\xi_{k,i}} (\bm{x}_k)$ would be difficult in practice. However, the previous numerical results in \cite{umeda2025increasing} showed that increasing the batch size a finite number of times (e.g., for training ResNet-34 on ImageNet with $n = 1,281,167$, $b_0 = 2^5, b_1 = 2^6, b_2 = 2^7$, $b_3 = 2^8$, and $b_4 = 2^9$ were used in Section A.5 in \cite{umeda2025increasing}) accelerates mini-batch SGD. From a practical standpoint, finitely increasing the batch size within a computable range leads to faster convergence of algorithm \eqref{mini_batch_KM}.

\subsection*{Notation and definitions}
Let $\mathbb{N}$ be the set of all natural numbers and let $[n] \coloneqq \{1,2,\cdots, n\}$ for $n \in \mathbb{N}$. Furthermore, let $\mathbb{R}^d$ be a $d$-dimensional Euclidean space with inner product $\langle \cdot, \cdot \rangle$ and norm $\|\cdot\|$ and $\mathbb{R}_+^d \coloneqq \{ \bm{x} = (x_i)_{i=1}^d \in \mathbb{R}^d \colon x_i \geq 0 \text{ } (i\in [n]) \}$. The identity mapping on $\mathbb{R}^d$ is denoted by $\mathrm{Id}$ (i.e., $\mathrm{Id}(\bm{x}) \coloneqq \bm{x}$ for all $\bm{x} \in \mathbb{R}^d$). A closed ball at center $\bm{x} \in \mathbb{R}^d$ with radius $r > 0$ is denoted by $B_r (\bm{x}) \coloneqq \{ \bm{y} \in \mathbb{R}^d \colon \|\bm{y} - \bm{x} \| \leq r \}$.

For a mapping $T \colon \mathbb{R}^d \to \mathbb{R}^d$, the set of all fixed points of $T$ is defined by $\mathrm{Fix} (T) \coloneqq \{ \bm{x} \in \mathbb{R}^d \colon \bm{x} = T(\bm{x}) \}$. We call it the fixed point set of $T$. $T$ is said to be nonexpansive if $\| T(\bm{x}) - T(\bm{y})\| \leq \|\bm{x} - \bm{y}\|$ for all $\bm{x}, \bm{y} \in \mathbb{R}^d$. The fixed point set of a nonexpansive mapping is closed and convex \cite[Proposition 5.3]{goebel1}. The metric projection $P_C$ onto a nonempty, closed convex set $C$ $(\subset \mathbb{R}^d)$ is defined for all $\bm{x} \in \mathbb{R}^d$ by $P_C (\bm{x}) \in C$ and $\| \bm{x} - P_C (\bm{x}) \| = \inf_{\bm{y} \in \mathbb{R}^d} \| \bm{x} - \bm{y} \|$. $P_C$ is nonexpansive with $\mathrm{Fix}(P_C) = C$ \cite[Theorem 3.1.4(i)]{takahashi}, \cite[p.371]{bau}, \cite[Theorem 2.4-3]{stark}.

For an operator $A \colon \mathbb{R}^d \to \mathbb{R}^d$, the set of all zero points of $A$ is defined by $A^{-1}(\bm{0}) \coloneqq \{ \bm{x} \in \mathbb{R}^d \colon A(\bm{x}) = \bm{0} \}$. We call it the zero point set of $A$. Let $\gamma > 0$. $A$ is said to be $\gamma$-inverse-strongly monotone ($\gamma$-cocoercive) \cite[Definition 4.10]{b-c} if $\langle \bm{x} - \bm{y}, A(\bm{x}) - A(\bm{y}) \rangle \geq \gamma \|A(\bm{x}) - A(\bm{y})\|^2$ for all $\bm{x}, \bm{y} \in \mathbb{R}^d$. Let $\beta \in (0,2 \gamma]$. When $A$ is $\gamma$-inverse-strongly monotone, a mapping $T \coloneqq \mathrm{Id} - \beta A$ is nonexpansive with $\mathrm{Fix}(T) = A^{-1}(\bm{0})$ \cite[Proposition 2.3]{iiduka_JOTA}.

For a function $f \colon \mathbb{R}^d \to \mathbb{R}$, the set of all minimizers of $f$ is denoted by $\argmin_{\bm{x} \in \mathbb{R}^d} f(\bm{x})$. The gradient of a differentiable function $f$ is denoted by $\nabla f$. Let $L > 0$, and let $f \colon \mathbb{R}^d \to \mathbb{R}$ be differentiable. $\nabla f$ is said to be $1/L$-Lipschitz continuous (i.e., $f$ is said to be $1/L$-smooth) if $\|\nabla f (\bm{x}) - \nabla f (\bm{y}) \| \leq (1/L) \| \bm{x} - \bm{y} \|$ for all $\bm{x}, \bm{y} \in \mathbb{R}^d$. When $f$ is convex with a $1/L$-Lipschitz continuous gradient, $\nabla f$ is $L$-inverse-strongly monotone \cite[Th\'{e}or\`{e}me 5]{baillon1977}. Hence, from \cite[Proposition 2.3]{iiduka_JOTA}, a mapping $T \coloneqq \mathrm{Id} - \eta \nabla f$, where $\eta \in (0,2L]$, is nonexpansive with $\mathrm{Fix}(T) = \argmin_{\bm{x} \in \mathbb{R}^d} f(\bm{x})$.

$\mathrm{P}(A)$ denotes the probability of event $A$. $\mathbb{E}_\xi [\bm{X}(\xi)]$ denotes the expectation of a random variable $\bm{X}(\xi)$ with respect to a random variable $\xi$. The variance of $\bm{X}(\xi)$ with respect to $\xi$ is defined by $\mathbb{V}_\xi [\bm{X}(\xi)] \coloneqq \mathbb{E}_\xi [\| \bm{X}(\xi) - \mathbb{E}_\xi [\bm{X}(\xi)] \|^2]$. $\mathbb{E}_\xi [\bm{X}(\xi)|\bm{Y}]$ (resp. $\mathbb{V}_\xi [\bm{X}(\xi)|\bm{Y}]$) denotes the expectation (resp. the variance) of $\bm{X}(\xi)$ conditioned on $\bm{Y}$. In the case of the independence of $\bm{\xi}_0, \bm{\xi}_1, \cdots, \bm{\xi}_k$, we define the total expectation $\mathbb{E}$ by $\mathbb{E} \coloneqq \mathbb{E}_{\bm{\xi}_0} \mathbb{E}_{\bm{\xi}_1} \cdots \mathbb{E}_{\bm{\xi}_k}$. In particular, we write $\xi \sim \mathrm{DU}(n)$ when $\xi$ follows a discrete uniform distribution on $[n]$.

\section{Stochastic Nonexpansive Fixed Point Problem}
\label{sec:2}
\subsection{Stochastic fixed point problem and examples}
Let us consider the following problem.

\begin{problem}\label{prob:1}
Let $n \in \mathbb{N}$ and $T_i \colon \mathbb{R}^d \to \mathbb{R}^d$ ($i\in [n]$). Furthermore, let $\xi$ be a random variable taking values in $[n]$. A stochastic mapping $T_\xi \colon \mathbb{R}^d \to \mathbb{R}^d$ is randomly chosen in $\{ T_i \}_{i=1}^{n}$ and a mapping $T \colon \mathbb{R}^d \to \mathbb{R}^d$ is defined for all $\bm{x} \in \mathbb{R}^d$ by 
\begin{align}\label{true_t}
T (\bm{x}) \coloneqq \mathbb{E}_\xi [T_\xi (\bm{x})],
\end{align}
where $\xi$ is independent of $\bm{x}$.
Here, we would like to find a fixed point $\bm{x}^\star$ of $T$, i.e., 
\begin{align*}
\bm{x}^\star = T (\bm{x}^\star) = \mathbb{E}_\xi [T_\xi (\bm{x}^\star)].
\end{align*} 
\end{problem}

This paper considers Problem \ref{prob:1} under the following assumptions:

\begin{assumption}\label{assum:1}
Let $n$ and $T$ be as in Problem \ref{prob:1}.
\begin{enumerate}
\item[{\em (A1)}] {\em [Nonexpansivity]} $T_i$ $(i\in [n])$ is nonexpansive, i.e., $\|T_i (\bm{x}) - T_i (\bm{y})\| \leq \|\bm{x} - \bm{y}\|$ $(\bm{x}, \bm{y} \in \mathbb{R}^d)$;
\item[{\em (A2)}] {\em [Boundedness of variance of stochastic mapping]} There exists $\sigma \geq 0$ such that, for all $\bm{x} \in \mathbb{R}^{d}$, 
\begin{align*}
\mathbb{V}_{\xi}[T_{\xi}(\bm{x})] \coloneqq \mathbb{E}_{\xi}[\| T_{\xi}(\bm{x}) - \underbrace{\mathbb{E}_{\xi}[T_{\xi}(\bm{x})]}_{T (\bm{x})} \|^2] \leq \sigma^{2},
\end{align*}
where $\xi$ is independent of $\bm{x}$.
\end{enumerate}
\end{assumption}

We give some examples of Problem \ref{prob:1} under Assumption \ref{assum:1}.

\begin{example}\label{exp:1}
Let $n \in \mathbb{N}$ and $\gamma_i > 0$ $(i\in [n])$, $\gamma \coloneqq \min_{i\in [n]} \gamma_i$, and $\beta \in (0,2 \gamma]$. Furthermore, let $L_i > 0$ $(i\in [n])$, $L \coloneqq \min_{i\in [n]} L_i$, and $\eta \in (0,2L]$. Let $C_i \subset \mathbb{R}^d$ $(i\in [n])$ be nonempty, closed, and convex with $C \coloneqq \bigcap_{i=1}^n C_i \neq \emptyset$, and let $P_{C_i} \coloneqq P_i$. Let $A_i \colon \mathbb{R}^d \to \mathbb{R}^d$ $(i\in [n])$ be $\gamma_i$-inverse-strongly monotone with $\bigcap_{i=1}^n A_i^{-1}(\bm{0}) \neq \emptyset$. Finally, let $f_i \colon \mathbb{R}^d \to \mathbb{R}$ $(i\in [n])$ be convex and $1/L_i$-smooth and $f \coloneqq (1/n) \sum_{i=1}^n f_i$.
\begin{enumerate}
\item[(i)] [Convex feasibility problem] Problem \ref{prob:1} for a mapping $T$ defined for all $\bm{x} \in \mathbb{R}^d$ by
\begin{align}\label{cfp_t}
T (\bm{x})
= 
\mathbb{E}_{\xi \sim \mathrm{DU}(n)} [P_\xi (\bm{x})]
= \frac{1}{n} \sum_{i=1}^n P_i (\bm{x})
\end{align}
is a convex feasibility problem for $C$ (that is, the problem of finding a point in $C$). Moreover, the following hold.

{(A1)} $T_i = P_i$ is nonexpansive.

{(A2)} for all $\bm{x}^\star \in C$ and all $\bm{x} \in B_{r}(\bm{x}^\star)$, $\mathbb{V}_{\xi \sim \mathrm{DU}(n)}[P_\xi (\bm{x})] \leq (r + \|\bm{x}^\star\|)^2$.

\item[(ii)] [Zero point problem] Problem \ref{prob:1} for a mapping $T$ defined for all $\bm{x} \in \mathbb{R}^d$ by
\begin{align}\label{zero_t}
T (\bm{x})
= 
\mathbb{E}_{\xi \sim \mathrm{DU}(n)} [(\mathrm{Id} - \beta A_\xi) (\bm{x})]
= \bm{x} - \frac{\beta}{n} \sum_{i=1}^n A_i (\bm{x})
\end{align}
is a zero point problem (that is, the problem of finding $\bm{x}^\star \in \bigcap_{i=1}^n A_i^{-1}(\bm{0})$). Moreover, the following hold.

{(A1)} $T_i = \mathrm{Id} - \beta A_i$ is nonexpansive.

{(A2)} for all $\bm{x}^\star \in \bigcap_{i=1}^n A_i^{-1}(\bm{0})$ and all $\bm{x} \in B_{r}(\bm{x}^\star)$, $\mathbb{V}_{\xi \sim \mathrm{DU}(n)}[\bm{x} - \beta A_\xi (\bm{x})] \leq r^2 (\beta (2 \gamma - \beta))^{-1}$. 

\item[(iii)] [Convex minimization problem] Let $\nabla f_\xi \colon \mathbb{R}^d \to \mathbb{R}^d$ be a stochastic gradient of $f = \frac{1}{n} \sum_{i=1}^n f_i$ such that 
\begin{enumerate}
\item[(C1)] [Unbiasedness of stochastic gradient] For all $\bm{x} \in \mathbb{R}^d$, $\mathbb{E}_{\xi \sim \mathrm{DU}(n)} [\nabla f_\xi (\bm{x})] = \nabla f (\bm{x})$;
\item[(C2)]
[Boundedness of variance of stochastic gradient] There exists $\sigma_g \geq 0$ such that, for all $\bm{x} \in \mathbb{R}^{d}$, $\mathbb{V}_{\xi \sim \mathrm{DU}(n)}[\nabla f_{\xi}(\bm{x})] \leq \sigma_g^{2}$.
\end{enumerate} 
Problem \ref{prob:1} for a mapping $T$ defined for all $\bm{x} \in \mathbb{R}^d$ by
\begin{align}\label{convex_t}
T (\bm{x})
= 
\mathbb{E}_{\xi \sim \mathrm{DU}(n)} [(\mathrm{Id} - \eta \nabla f_\xi) (\bm{x})]
= \bm{x} - \eta \nabla f(\bm{x})
\end{align}
is a convex minimization problem for $f$. Moreover, the following hold.

{(A1)} $T_i = \mathrm{Id} - \eta \nabla f_i$ is nonexpansive.

{(A2)} for all $\bm{x} \in \mathbb{R}^d$, $\mathbb{V}_{\xi \sim \mathrm{DU}(n)}^{\mathfrak{p}}[\bm{x} - \eta \nabla f_\xi (\bm{x})] \leq (\eta \sigma_g)^{2}$.

In particular, $\sigma_g = \sqrt{ (2/n) \sum_{i=1}^n (f_i^{\star\star} - f_i^{\star}) L_i^{-1}}$ on $\{ \bm{x} \in \mathbb{R}^d \colon - \infty < f_i^\star \leq f_i (\bm{x}) \leq f_i^{\star\star} < + \infty \text{ } (i\in [n])\}$.
\end{enumerate} 
\end{example}

\begin{proof}
(i) We define $T_i \coloneqq P_i$ $(i\in [n])$. Under $C \coloneqq \bigcap_{i=1}^n C_i = \bigcap_{i=1}^n\mathrm{Fix}(T_i) \neq \emptyset$, Proposition 4.47 in \cite{b-c} and \eqref{cfp_t} ensure that 
\begin{align*}
\mathrm{Fix}(T) 
= \mathrm{Fix} \left( \frac{1}{n} \sum_{i=1}^n P_i  \right)
= \mathrm{Fix} \left( \frac{1}{n} \sum_{i=1}^n T_i  \right)
= \bigcap_{i=1}^n \mathrm{Fix}(T_i)
= \bigcap_{i=1}^n C_i = C,
\end{align*}
which implies that Problem \ref{prob:1} for $T$ defined by \eqref{cfp_t} is a convex feasibility problem. Since $P_i$ is nonexpansive (see, e.g., \cite[Theorem 3.1.4]{takahashi}), (A1) holds.

We prove that (A2) holds. The triangle inequality and the nonexpansivity of $P_i$ $(i\in [n])$ imply that, for all $\bm{x}^\star \in C$ and all $\bm{x} \in B_{r} (\bm{x}^\star)$,
\begin{align}\label{norm_p_i}
\|P_i (\bm{x})\| 
\leq 
\| P_i (\bm{x}) - P_i (\bm{x}^\star)  \| + \|\bm{x}^\star\|
\leq \|\bm{x} - \bm{x}^\star\| + \|\bm{x}^\star\| \leq r + \|\bm{x}^\star\|.
\end{align}
Using \eqref{norm_p_i} leads to the finding that 
\begin{align*}
\mathbb{V}_{\xi \sim \mathrm{DU}(n)}[P_\xi (\bm{x})]
&\leq 
\mathbb{E}_{\xi \sim \mathrm{DU}(n)} \left[\| P_\xi (\bm{x})\|^2 \right]
= 
\sum_{i=1}^n \|P_i (\bm{x})\|^2 \mathrm{P}(\xi = i)
=  
\left( r + \|\bm{x}^\star\| \right)^{2}.
\end{align*}

(ii) We define $T_i \coloneqq \mathrm{Id} - \beta A_i$ $(i\in [n])$. Under $\bigcap_{i=1}^n A_i^{-1}(\bm{0}) = \bigcap_{i=1}^n\mathrm{Fix}(T_i) \neq \emptyset$, Proposition 4.47 in \cite{b-c} and \eqref{zero_t} ensure that 
\begin{align*}
\mathrm{Fix}(T) 
= \mathrm{Fix} \left( \frac{1}{n} \sum_{i=1}^n (\mathrm{Id} - \beta A_i)  \right)
= \mathrm{Fix} \left( \frac{1}{n} \sum_{i=1}^n T_i  \right)
= \bigcap_{i=1}^n \mathrm{Fix}(T_i)
= \bigcap_{i=1}^n A_i^{-1}(\bm{0}),
\end{align*}
which implies that Problem \ref{prob:1} for $T$ defined by \eqref{zero_t} is a zero point problem. From $\beta \leq 2 \gamma \leq 2 \gamma_i$ $(i\in [n])$ and \cite[Proposition 2.3]{iiduka_JOTA}, $\mathrm{Id} - \beta A_i$ is nonexpansive, which implies that (A1) holds.

We prove that (A2) holds. The $\gamma_i$-inverse-strong monotonicity of $A_i$ $(i\in [n])$ and the relation $\|\bm{x} - \bm{y} \|^2 = \|\bm{x}\|^2 -2 \langle \bm{x},\bm{y} \rangle + \| \bm{y} \|^2$ $(\bm{x},\bm{y} \in \mathbb{R}^d)$ ensure that, for all $\bm{x}^\star \in \bigcap_{i=1}^n A_i^{-1}(\bm{0})$ and all $\bm{x} \in B_{r}(\bm{x}^\star)$, 
\begin{align*}
0 
&\leq 
\|( \bm{x} - \beta A_i (\bm{x})) - ( \bm{x}^\star - \beta A_i (\bm{x}^\star))\|^2
= 
\| (\bm{x} - \bm{x}^\star) - \beta (A_i (\bm{x}) - A_i (\bm{x}^\star)) \|^2\\
&= 
\| \bm{x} - \bm{x}^\star \|^2 - 2 \beta \langle \bm{x} - \bm{x}^\star, A_i (\bm{x}) - A_i (\bm{x}^\star) \rangle 
+ \beta^2 \|A_i (\bm{x}) - A_i (\bm{x}^\star)\|^2\\
&\leq
\| \bm{x} - \bm{x}^\star \|^2 - 2 \beta \gamma_i \| A_i (\bm{x}) - A_i (\bm{x}^\star) \|^2 
+ \beta^2 \|A_i (\bm{x}) - A_i (\bm{x}^\star)\|^2\\
&= 
\| \bm{x} - \bm{x}^\star \|^2 + \beta (\beta - 2 \gamma_i)\|A_i (\bm{x}) - A_i (\bm{x}^\star)\|^2,
\end{align*}
which, together with $\gamma = \min_{i \in [n]} \gamma_i$, $\beta \in (0,2 \gamma]$, $\bm{x} \in B_r (\bm{x}^\star)$, and $A_i (\bm{x}^\star) = \bm{0}$, implies that 
\begin{align}\label{norm_a_i}
\|A_i (\bm{x})\|^2 \leq \frac{r^2}{\beta (2 \gamma - \beta)}.
\end{align}
From \eqref{norm_a_i}, we have 
\begin{align*}
\mathbb{V}_{\xi \sim \mathrm{DU}(n)}[A_\xi (\bm{x})]
&\leq 
\mathbb{E}_{\xi \sim \mathrm{DU}(n)} \left[\| A_\xi (\bm{x})\|^2 \right] 
= 
\sum_{i=1}^n \|A_i (\bm{x})\|^2 \mathrm{P}(\xi = i)
=
\frac{r^2}{\beta (2 \gamma - \beta)}.
\end{align*} 

(iii) We define $T_i \coloneqq \mathrm{Id} - \eta \nabla f_i$ $(i\in [n])$. From \eqref{convex_t} and the convexity of $f$, we have 
\begin{align*}
\mathrm{Fix}(T) = (\nabla f)^{-1} (\bm{0}) = \argmin_{\bm{x} \in \mathbb{R}^d} f(\bm{x}),
\end{align*}
which implies that Problem \ref{prob:1} for $T$ defined by \eqref{convex_t} is a convex minimization problem. $\nabla f_i$ is $L_i$-inverse-strongly monotone \cite[Th\'{e}or\`{e}me 5]{baillon1977}. From $\eta \leq 2 L \leq 2 L_i$ $(i\in [n])$ and \cite[Proposition 2.3]{iiduka_JOTA}, $\mathrm{Id} - \eta \nabla f_i$ is nonexpansive, which implies that (A1) holds.

Finally, we prove that (A2) holds. Using (C1) and (C2) leads to 
\begin{align*}
\mathbb{V}_{\xi \sim \mathrm{DU}(n)}[\bm{x} - \eta \nabla f_\xi (\bm{x})]
&= 
\mathbb{E}_{\xi \sim \mathrm{DU}(n)}
[ \| ( \bm{x} - \eta \nabla f_\xi (\bm{x}) ) - ( \bm{x} - \eta \nabla f (\bm{x}) ) \|^{2} ]\\
&= 
\eta^{2}
\mathbb{E}_{\xi \sim \mathrm{DU}(n)} [ \|\nabla f_\xi (\bm{x}) - \nabla f (\bm{x})\|^{2}]\\
&=
\eta^{2} \mathbb{V}_{\xi \sim \mathrm{DU}(n)}
[\nabla f_\xi (\bm{x})] 
\leq 
(\eta \sigma_g)^{2}.
\end{align*}
Let $\bm{x}$ satisfy $- \infty < f_i^\star \leq f_i (\bm{x}) \leq f_i^{\star\star} < + \infty$ and $\bm{y} \coloneqq \bm{x} - L_i \nabla f_i (\bm{x})$. The descent lemma for the $1/L_i$-smoothness of $f_i$ ensures that
\begin{align*}
f_i^\star 
&\leq f_i(\bm{y}) 
\leq 
f_i(\bm{x}) + \langle \nabla f_i (\bm{x}), \bm{y} - \bm{x} \rangle
+ 
\frac{1}{2 L_i} \|\bm{y} - \bm{x}\|^2\\
&= 
f_i(\bm{x})
- L_i \| \nabla f_i (\bm{x}) \|^2
+ \frac{L_i}{2} \| \nabla f_i (\bm{x}) \|^2\\
&= 
f_i(\bm{x}) - \frac{L_i}{2} \| \nabla f_i (\bm{x}) \|^2,
\end{align*}
which, together with $f_i^{\star\star} \in \mathbb{R}$, implies that
\begin{align}\label{norm_nabla_f_i}
\| \nabla f_i (\bm{x}) \|^{2}
\leq
\frac{2}{L_i} (f_i^{\star\star} - f_i^{\star}).
\end{align}
Using \eqref{norm_nabla_f_i} implies 
\begin{align*}
\mathbb{V}_{\xi \sim \mathrm{DU}(n)} \left[ \nabla f_\xi (\bm{x}) \right]
&\leq  
\mathbb{E}_{\xi \sim \mathrm{DU}(n)} \left[ \|\nabla f_\xi (\bm{x})\|^2   \right]
=
\sum_{i=1}^n \|\nabla f_i (\bm{x})\|^2 \mathrm{P}(\xi = i)\\
&= 
\frac{2}{n} \sum_{i=1}^n
\frac{f_i^{\star\star} - f_i^{\star}}{L_i},
\end{align*}
which complete the proof.
\qed
\end{proof}

\subsection{Mini-batch stochastic mapping}
Let $b \in \mathbb{N}$ be the batch size (number of samples), and let $\bm{\xi} = (\xi_{1}, \xi_{2}, \cdots, \xi_{b})^\top$ comprise $b$ i.i.d. variables. Here, we define the {\em mini-batch stochastic mapping} of $T$ obeying \eqref{true_t} for any $\bm{x} \in \mathbb{R}^d$ as follows:
\begin{align}\label{mini_batch}
T_{\bm{\xi}}(\bm{x}) \coloneqq 
\frac{1}{b} \sum_{i=1}^b T_{\xi_{i}}(\bm{x}),
\end{align} 
where $\bm{x}$ is independent of $\bm{\xi}$. The following proposition demonstrates that the mini-batch stochastic mapping $T_{\bm{\xi}}$ inherits useful properties of the stochastic mapping $T_\xi$, such as unbiasedness (\eqref{true_t} in Problem \ref{prob:1}) and boundedness of variance (Assumption \ref{assum:1}(A2)). 

\begin{proposition}\label{prop:1}
Consider Problem \ref{prob:1} under Assumption \ref{assum:1} and let $T_{\bm{\xi}}$ be the mini-batch stochastic mapping defined by (\ref{mini_batch}). Further, let $\bm{x}, \bm{y} \in \mathbb{R}^d$ be independent of $\bm{\xi}$. Then, the following hold.
\begin{enumerate}
\item[{\em (i)}] {\em [Unbiasedness of mini-batch stochastic mapping]} $\displaystyle{\mathbb{E}_{\bm{\xi}}[T_{\bm{\xi}}(\bm{x})] = T(\bm{x})}$;
\item[{\em (ii)}] {\em [Boundedness of variance of mini-batch stochastic mapping]} $\displaystyle{\mathbb{V}_{\bm{\xi}}[T_{\bm{\xi}} (\bm{x})] \leq \frac{\sigma^2}{b}}$.
\end{enumerate}
Moreover, {\em (i)} and {\em (ii)} lead to the finding that 
\begin{align}
&\mathbb{E}_{\bm{\xi}} \left[ \| T_{\bm{\xi}}(\bm{x}) - T(\bm{y}) \|^2 \right]
\leq 
\|\bm{x} - \bm{y} \|^2 
+ 
\frac{\sigma^2}{b}, \label{app_nonexp}\\
&
\|\bm{x} - T (\bm{x})\|^2
\leq
\mathbb{E}_{\bm{\xi}} \left[ \| \bm{x} - T_{\bm{\xi}}(\bm{x}) \|^2 \right] 
\leq \|\bm{x} - T (\bm{x})\|^2 + \frac{\sigma^2}{b}. \label{app_nonexp_1}
\end{align}
\end{proposition}

\begin{proof}
(i) From the property of $\mathbb{E}_{\bm{\xi}}$ and the definitions of $T_{\bm{\xi}}$ and $T$, we have 
\begin{align*}
\mathbb{E}_{\bm{\xi}}[T_{\bm{\xi}}(\bm{x})]
= 
\mathbb{E}_{\bm{\xi}} \left[
\frac{1}{b} \sum_{i=1}^b T_{\xi_{i}}(\bm{x})
\right]
= 
\frac{1}{b} \sum_{i=1}^b \mathbb{E}_{\xi_i} [T_{\xi_{i}}(\bm{x})]
= \frac{1}{b} \sum_{i=1}^b T (\bm{x})
= T (\bm{x}).
\end{align*}

(ii) The property of $\mathbb{V}_{\bm{\xi}}$ and the independence of $\xi_i$s ensure that
\begin{align*}
\mathbb{V}_{\bm{\xi}}[T_{\bm{\xi}} (\bm{x})]
&= 
\mathbb{V}_{\bm{\xi}} \left[
\frac{1}{b} \sum_{i=1}^b T_{\xi_{i}}(\bm{x}) 
\right]
= 
\frac{1}{b^{2}}
\mathbb{V}_{\bm{\xi}} \left[
\sum_{i=1}^b T_{\xi_{i}}(\bm{x})
\right]
= 
\frac{1}{b^{2}}
\sum_{i=1}^b
\mathbb{V}_{\xi_i} \left[
T_{\xi_{i}}(\bm{x})
\right],
\end{align*}
which, together with Assumption \ref{assum:1}(A2), implies that
\begin{align*}
\mathbb{V}_{\bm{\xi}}[T_{\bm{\xi}} (\bm{x})]
\leq 
\frac{1}{b^{2}}
\sum_{i=1}^b \sigma^2
= \frac{\sigma^2}{b}.
\end{align*}

Next, we prove that \eqref{app_nonexp} holds. The definition \eqref{true_t} of $T$ and the property of $\mathbb{E}_\xi$ ensure that
\begin{align*}
\| T(\bm{x}) - T (\bm{y})  \|
&= 
\| \mathbb{E}_\xi [T_\xi (\bm{x}) ] - \mathbb{E}_\xi [T_\xi (\bm{y}) ] \|
= 
\| \mathbb{E}_\xi [T_\xi (\bm{x}) - T_\xi (\bm{y}) ] \|\\
&= 
\left\| \sum_{i=1}^n (T_i (\bm{x}) - T_i (\bm{y}) ) \mathrm{P}(\xi = i) \right\|,
\end{align*}
which, together with the triangle inequality and Assumption \ref{assum:1}(A1), implies that 
\begin{align}\label{nonexp_T}
\| T(\bm{x}) - T (\bm{y})  \|
\leq
\sum_{i=1}^n \| T_i (\bm{x}) - T_i (\bm{y}) \| \mathrm{P}(\xi = i)
\leq 
\| \bm{x} - \bm{y} \|.
\end{align}
That is, $T$ defined by \eqref{true_t} is nonexpansive. From the equation $\|\bm{x} + \bm{y}\|^2 = \|\bm{x}\|^2 + 2 \langle \bm{x},\bm{y} \rangle + \|\bm{y}\|^2$ $(\bm{x},\bm{y} \in \mathbb{R}^d)$ and the property of $\mathbb{E}_{\bm{\xi}}$,
\begin{align}\label{ineq:1}
\begin{split}
&\mathbb{E}_{\bm{\xi}} \left[ \| T_{\bm{\xi}}(\bm{x}) - T(\bm{y}) \|^2 \right]\\
&= 
\mathbb{E}_{\bm{\xi}} \left[ \| (T_{\bm{\xi}}(\bm{x}) - T(\bm{x})) 
+ (T(\bm{x}) - T (\bm{y})) \|^2 \right]\\
&=
\mathbb{E}_{\bm{\xi}} \left[ \| T_{\bm{\xi}}(\bm{x}) - T(\bm{x}) \|^2 \right]
+ 2 \mathbb{E}_{\bm{\xi}} \left[ \langle T_{\bm{\xi}}(\bm{x}) - T(\bm{x}),  T(\bm{x}) - T (\bm{y}) \rangle \right]\\
&\quad + 
\mathbb{E}_{\bm{\xi}} \left[ \| T(\bm{x}) - T (\bm{y})  \|^2 \right]\\
&= 
\underbrace{\mathbb{E}_{\bm{\xi}} \left[ \| T_{\bm{\xi}}(\bm{x}) - T(\bm{x}) \|^2 \right]}_{\mathbb{V}_{\bm{\xi}}[T_{\bm{\xi}}(\bm{x})] \leq \frac{\sigma^2}{b}}
+ 2  \langle \underbrace{\mathbb{E}_{\bm{\xi}} \left[ T_{\bm{\xi}}(\bm{x})\right]}_{T(\bm{x})} - T(\bm{x}),  T(\bm{x}) - T (\bm{y}) \rangle \\
&\quad + 
\| T(\bm{x}) - T (\bm{y})  \|^2,
\end{split} 
\end{align}
which, together with the nonexpansivity of $T$ and Proposition \ref{prop:1}(i) and (ii), implies that \eqref{app_nonexp} holds.

Finally, we prove that \eqref{app_nonexp_1} holds. A similar argument to the one above for \eqref{ineq:1} leads to the finding that 
\begin{align*}
&\mathbb{E}_{\bm{\xi}} \left[ \| \bm{x} - T_{\bm{\xi}}(\bm{x}) \|^2 \right]\\
&= 
\mathbb{E}_{\bm{\xi}} \left[ \| (\bm{x} - T(\bm{x})) 
+ (T(\bm{x}) - T_{\bm{\xi}} (\bm{x})) \|^2 \right]\\
&=
\mathbb{E}_{\bm{\xi}} \left[ \| \bm{x} - T(\bm{x}) \|^2 \right]
+ 2 \mathbb{E}_{\bm{\xi}} \left[ \langle \bm{x} - T(\bm{x}), T(\bm{x}) - T_{\bm{\xi}} (\bm{x}) \rangle \right]\\
&\quad + 
\mathbb{E}_{\bm{\xi}} \left[ \| T(\bm{x}) - T_{\bm{\xi}} (\bm{x})  \|^2 \right]\\
&=
\| \bm{x} - T(\bm{x}) \|^2
+ 2  \langle \bm{x} - T(\bm{x}), T(\bm{x}) - \underbrace{\mathbb{E}_{\bm{\xi}} \left[T_{\bm{\xi}} (\bm{x}) \right]}_{T(\bm{x})} \rangle 
+ 
\underbrace{\mathbb{E}_{\bm{\xi}} \left[ \| T(\bm{x}) - T_{\bm{\xi}} (\bm{x})  \|^2 \right]}_{0 \leq \mathbb{V}_{\bm{\xi}}[T_{\bm{\xi}} (\bm{x})] \leq \frac{\sigma^2}{b}}.
\end{align*}
This completes the proof.
\qed
\end{proof}

Let us check the properties of \eqref{app_nonexp} and \eqref{app_nonexp_1}. Inequality \eqref{app_nonexp} implies that, when the batch size $b$ is large enough,
\begin{align*}
&\mathbb{E}_{\bm{\xi}} \left[ \| T_{\bm{\xi}}(\bm{x}) - T(\bm{x}) \|^2 \right]
\leq 
\|\bm{x} - \bm{x} \|^2 
+ 
\frac{\sigma^2}{b}
\approx 
0,\\
&\mathbb{E}_{\bm{\xi}} \left[ \| T_{\bm{\xi}}(\bm{x}) - T(\bm{y}) \|^2 \right]
\leq 
\|\bm{x} - \bm{y} \|^2 
+ 
\frac{\sigma^2}{b}
\approx 
\|\bm{x} - \bm{y} \|^2,
\end{align*}
which implies that $T_{\bm{\xi}}$ with a large $b$ approximates a nonexpansive mapping $T$ (see \eqref{nonexp_T} for the nonexpansivity of $T$). Inequality \eqref{app_nonexp_1} implies that, when the batch size $b$ is large enough,
\begin{align*} 
\|\bm{x} - T (\bm{x})\|^2
\leq
\mathbb{E}_{\bm{\xi}} \left[ \| \bm{x} - T_{\bm{\xi}}(\bm{x}) \|^2 \right] 
\leq \|\bm{x} - T (\bm{x})\|^2 + \frac{\sigma^2}{b} \approx \|\bm{x} - T (\bm{x})\|^2.
\end{align*}
In particular, setting $\bm{x} = \bm{x}^\star \in \mathrm{Fix}(T)$ leads to
\begin{align*}
\mathbb{E}_{\bm{\xi}} \left[ \| \bm{x}^\star - T_{\bm{\xi}}(\bm{x}^\star) \|^2 \right] 
\approx \|\bm{x}^\star - T (\bm{x}^\star)\|^2 = 0,
\end{align*}
which implies that a fixed point of $T$ can be approximately a fixed point of $T_{\bm{\xi}}$ with a large $b$.

\section{Mini-Batch Stochastic Krasnosel'ki\u\i-Mann Algorithm}
\label{sec:3}
Let $\xi_{k,i}$ be a random variable generated by the $i$-th sampling in the $k$-th iteration. Since the $k$-th iterate $\bm{x}_k$ is computed before $\bm{\xi}_k = (\xi_{k,1}, \xi_{k,2}, \cdots, \xi_{k,b_k})^\top$ is sampled, $\bm{x}_k$ is independent of $\bm{\xi}_k$. Hence, by referring to \eqref{mini_batch}, we can define the mini-batch stochastic mapping of $T$ at the $k$-th iteration by
\begin{align}\label{mini_batch_t}
T_{\bm{\xi}_k} (\bm{x}_k) \coloneqq \frac{1}{b_k} \sum_{i=1}^{b_k} T_{\xi_{k,i}}(\bm{x}_k).
\end{align}
The pseudo-code of the mini-batch stochastic Krasnosel'ki\u\i-Mann Algorithm is listed below. 

\begin{algorithm} 
\caption{Mini-Batch Stochastic Krasnosel'ki\u\i-Mann Algorithm} 
\label{algo:1} 
\begin{algorithmic}[1] 
\REQUIRE
$\bm{x}_{0} \in\mathbb{R}^d$ (initial point), $\alpha_k \in (0,1)$ (step size), $b_k \in \mathbb{N}$ (batch size), $K \in \mathbb{N}$ (steps).
\ENSURE
$\bm{x}_{K}$
\FOR{$k = 0, 1, \cdots, K-1$}
\STATE
$\bm{\xi}_k = (\xi_{k,1}, \xi_{k,2}, \cdots, \xi_{k,b_k})^\top$
\STATE
$T_{\bm{\xi}_k}(\bm{x}_k) \coloneqq \frac{1}{b_k} \sum_{i=1}^{b_k} T_{\xi_{k,i}}(\bm{x}_k)$
\STATE
$\bm{x}_{k+1} \coloneqq 
\bm{x}_k + \alpha_k (T_{\bm{\xi}_k}(\bm{x}_k) - \bm{x}_k)
=
(1 - \alpha_k) \bm{x}_k + \alpha_k T_{\bm{\xi}_k}(\bm{x}_k)$
\STATE
$k \gets k + 1$
\ENDFOR
\end{algorithmic}
\end{algorithm}

We may in theory assume sampling with replacement. In sampling with replacement, even if the batch size $b_k$ exceeds $n$, $T_{\bm{\xi}_k} \neq T$ holds in general. Hence, to examine the convergence of mini-batch algorithms under sampling with replacement, we can use $b_k \to + \infty$ $(k \to + \infty)$.

\subsection{Convergence}
\label{sec:3.1}
Two lemmas, Lemmas \ref{lem:1} and \ref{lem:2}, are needed to show almost-sure convergence of Algorithm \ref{algo:1} (Theorem \ref{thm:1}). Lemma \ref{lem:1} indicates that, under Assumption \ref{assum:1}, the inferior limit of $\|\bm{x}_k - T (\bm{x}_k)\|$ is $0$ almost surely. 

\begin{lemma}\label{lem:1}
Under Assumption \ref{assum:1}, the sequence $(\bm{x}_k)$ generated by Algorithm \ref{algo:1} satisfies that, for all $k \in \{0\} \cup \mathbb{N}$ and all $\bm{x}^\star \in \mathrm{Fix}(T)$,
\begin{align*}
\mathbb{E}_{\bm{\xi}_k} \left[ \| \bm{x}_{k+1} - \bm{x}^\star \|^2 \big| \bm{\xi}_{[k-1]} \right]
\leq 
\| \bm{x}_{k} - \bm{x}^\star \|^2 + \frac{\sigma^2 \alpha_k}{b_k}
- \alpha_k (1 - \alpha_k) \| \bm{x}_k - T (\bm{x}_{k})\|^2,
\end{align*}
where $\bm{\xi}_{[k-1]} \coloneqq \{\bm{\xi}_0, \bm{\xi}_1, \cdots, \bm{\xi}_{k-1} \}$, which implies 
\begin{align*}
\sum_{k=0}^{+ \infty} \frac{\alpha_k}{b_k} < + \infty
\Rightarrow 
\begin{dcases}
\exists \lim_{k \to + \infty} \|\bm{x}_k - \bm{x}^\star \| \text{ a.s.}\\
\sum_{k=0}^{+\infty} \alpha_k (1 - \alpha_k) \| \bm{x}_k - T (\bm{x}_{k})\|^2 < + \infty \text{ a.s.}.
\end{dcases}
\end{align*}
Moreover, 
\begin{align*}
\sum_{k=0}^{+\infty} \alpha_k (1 - \alpha_k) = + \infty
\Rightarrow 
\liminf_{k \to + \infty} \|\bm{x}_k - T (\bm{x}_k) \| = 0 \text{ a.s.}.
\end{align*}
\end{lemma}

\begin{proof}
The equation $\|\alpha \bm{x} + (1-\alpha) \bm{y} \|^2 = \alpha \|\bm{x}\|^2 + (1-\alpha) \|\bm{y}\|^2 - \alpha (1-\alpha)\|\bm{x} - \bm{y}\|^2$ $(\bm{x},\bm{y} \in \mathbb{R}^d, \alpha \in \mathbb{R})$ and the definition of $\bm{x}_{k+1}$ (Step 4 of Algorithm \ref{algo:1}) imply that
\begin{align*}
&\mathbb{E}_{\bm{\xi}_k} \left[ \| \bm{x}_{k+1} - \bm{x}^\star \|^2 \big| \bm{\xi}_{[k-1]} \right]\\
&= 
\mathbb{E}_{\bm{\xi}_k} \left[ \| (1-\alpha_k) (\bm{x}_{k} - \bm{x}^\star)
+ \alpha_k (T_{\bm{\xi}_k} (\bm{x}_{k}) - \bm{x}^\star) \|^2 \big| \bm{\xi}_{[k-1]} \right]\\
&=
(1-\alpha_k) \mathbb{E}_{\bm{\xi}_k} \left[ \| \bm{x}_{k} - \bm{x}^\star \|^2 \big| \bm{\xi}_{[k-1]} \right]
+ 
\alpha_k \mathbb{E}_{\bm{\xi}_k} \left[ \| T_{\bm{\xi}_k} (\bm{x}_{k}) - \bm{x}^\star \|^2 \big| \bm{\xi}_{[k-1]} \right]\\
&\quad - \alpha_k (1 - \alpha_k) \mathbb{E}_{\bm{\xi}_k} \left[ \| \bm{x}_k - T_{\bm{\xi}_k} (\bm{x}_{k})\|^2 \big| \bm{\xi}_{[k-1]} \right]. 
\end{align*}
Meanwhile, Proposition \ref{prop:1} with $\bm{y} = \bm{x}^\star = T(\bm{x}^\star)$ ensures that
\begin{align}\label{ineq_prop}
\begin{split}
&\mathbb{E}_{\bm{\xi}_k} \left[ \| T_{\bm{\xi}_k} (\bm{x}_{k}) - \bm{x}^\star \|^2 \big| \bm{\xi}_{[k-1]} \right]
\leq 
\| \bm{x}_{k} - \bm{x}^\star \|^2 + \frac{\sigma^2}{b_k},\\
&\| \bm{x}_k - T (\bm{x}_{k})\|^2
\leq
\mathbb{E}_{\bm{\xi}_k} \left[ \| \bm{x}_k - T_{\bm{\xi}_k} (\bm{x}_{k})\|^2 \big| \bm{\xi}_{[k-1]} \right]
\leq 
\| \bm{x}_k - T (\bm{x}_{k})\|^2 + \frac{\sigma^2}{b_k}.
\end{split}
\end{align} 
$\alpha_k \in (0,1)$ and \eqref{ineq_prop} lead to the finding that
\begin{align*}
&\mathbb{E}_{\bm{\xi}_k} \left[ \| \bm{x}_{k+1} - \bm{x}^\star \|^2 \big| \bm{\xi}_{[k-1]} \right]\\
&\leq 
(1-\alpha_k)  \| \bm{x}_{k} - \bm{x}^\star \|^2 
+ 
\alpha_k \left( \| \bm{x}_{k} - \bm{x}^\star \|^2 + \frac{\sigma^2}{b_k} \right)
- \alpha_k (1 - \alpha_k) \| \bm{x}_k - T (\bm{x}_{k})\|^2\\
&= 
\| \bm{x}_{k} - \bm{x}^\star \|^2 + \frac{\sigma^2 \alpha_k}{b_k}
- \alpha_k (1 - \alpha_k) \| \bm{x}_k - T (\bm{x}_{k})\|^2.
\end{align*}
The super martingale convergence theorem \cite[Proposition 8.2.10]{bert} with $\sum_{k=0}^{+ \infty} \alpha_k/b_k < + \infty$ ensures the existence of the limit of $\|\bm{x}_k - \bm{x}^\star \|$ and $\sum_{k=0}^{+\infty} \alpha_k (1 - \alpha_k) \| \bm{x}_k - T (\bm{x}_{k})\|^2 < + \infty$ with probability $1$. Finally, we prove that, under $\sum_{k=0}^{+\infty} \alpha_k (1 - \alpha_k) = + \infty$ and $\sum_{k=0}^{+\infty} \alpha_k (1 - \alpha_k) \| \bm{x}_k - T (\bm{x}_{k})\|^2 < + \infty$ a.s., $\liminf_{k \to + \infty} \|\bm{x}_k - T (\bm{x}_k)\| = 0$ a.s. by contradiction. Assuming $\liminf_{k \to + \infty} \|\bm{x}_k - T (\bm{x}_k)\| > 0$ a.s. implies that there exist $r > 0$ and $k_0 \in \mathbb{N}$ such that, for all $k \geq k_0$, $\|\bm{x}_k - T (\bm{x}_k)\| \geq r$ a.s.. Then, 
\begin{align*}
+ \infty 
= 
r^2 \sum_{k=0}^{+\infty} \alpha_k (1 - \alpha_k)
\leq
\sum_{k=0}^{+\infty} \alpha_k (1 - \alpha_k) \| \bm{x}_k - T (\bm{x}_{k})\|^2
< + \infty,
\end{align*}
which is a contradiction. This completes the proof.
\qed
\end{proof}

Lemma \ref{lem:2} asserts the almost-sure existence of the limit of $\|\bm{x}_k - T(\bm{x}_k)\|$.

\begin{lemma}\label{lem:2}
Under Assumption \ref{assum:1}, the sequence $(\bm{x}_k)$ generated by Algorithm \ref{algo:1} satisfies that, for all $k \in \{0\} \cup \mathbb{N}$,
\begin{align*}
\mathbb{E}_{\bm{\xi}_k} \left[ \| \bm{x}_{k+1} - T (\bm{x}_{k+1}) \| \big| \bm{\xi}_{[k-1]} \right]
\leq
\| \bm{x}_{k} - T (\bm{x}_{k}) \|
+ 
\frac{2 \sigma}{\sqrt{b_k}},
\end{align*} 
which implies 
\begin{align*}
\sum_{k=0}^{+ \infty} \frac{1}{\sqrt{b_k}} < + \infty
\Rightarrow
\exists \lim_{k \to + \infty} \| \bm{x}_k - T(\bm{x}_k)\| \text{ a.s.}.
\end{align*}
\end{lemma}

\begin{proof}
The triangle inequality implies that
\begin{align*}
&\mathbb{E}_{\bm{\xi}_k} \left[ \| \bm{x}_{k+1} - T (\bm{x}_{k+1}) \| \big| \bm{\xi}_{[k-1]} \right]\\
&=
\mathbb{E}_{\bm{\xi}_k} \left[ \| (1-\alpha_k) (\bm{x}_{k} - T (\bm{x}_{k+1})) + \alpha_k (T_{\bm{\xi}_k} (\bm{x}_{k}) - T (\bm{x}_{k+1})) \| \big| \bm{\xi}_{[k-1]} \right]\\
&\leq
(1-\alpha_k) \mathbb{E}_{\bm{\xi}_k} \left[ \| \bm{x}_{k} - T (\bm{x}_{k+1}) \| \big| \bm{\xi}_{[k-1]} \right]
+ \alpha_k \mathbb{E}_{\bm{\xi}_k} \left[ \| T_{\bm{\xi}_k} (\bm{x}_{k}) - T (\bm{x}_{k+1}) \| \big| \bm{\xi}_{[k-1]} \right]\\
&\leq
(1-\alpha_k) \mathbb{E}_{\bm{\xi}_k} \left[ \| \bm{x}_{k} - T (\bm{x}_{k+1}) \| \big| \bm{\xi}_{[k-1]} \right]
+ \alpha_k \underbrace{\mathbb{E}_{\bm{\xi}_k} \left[ \| T_{\bm{\xi}_k} (\bm{x}_{k}) - T (\bm{x}_{k}) \| \big| \bm{\xi}_{[k-1]} \right]}_{\leq \sqrt{\mathbb{V}_{\bm{\xi}_k} [ T_{\bm{\xi}_k} (\bm{x}_{k}) | \bm{\xi}_{[k-1]}]}}\\
&\quad 
+ \alpha_k 
\underbrace{\mathbb{E}_{\bm{\xi}_k} \left[ \| T (\bm{x}_{k}) - T (\bm{x}_{k+1}) \| \big| \bm{\xi}_{[k-1]} \right]}_{X_k},
\end{align*}
where 
\begin{align}\label{jensen}
\begin{split}
\left( \mathbb{E}_{\bm{\xi}_k} \left[ \| T_{\bm{\xi}_k} (\bm{x}_{k}) - T (\bm{x}_{k}) \| \big| \bm{\xi}_{[k-1]} \right] \right)^2 
&\leq 
\mathbb{E}_{\bm{\xi}_k} \left[ \| T_{\bm{\xi}_k} (\bm{x}_{k}) - T (\bm{x}_{k}) \|^2 \big| \bm{\xi}_{[k-1]} \right]\\
&= 
\mathbb{V}_{\bm{\xi}_k} \left[ T_{\bm{\xi}_k} (\bm{x}_{k}) \big| \bm{\xi}_{[k-1]} \right]
\end{split}
\end{align}
comes from Jensen's inequality and Proposition \ref{prop:1}(i). From the nonexpansivity \eqref{nonexp_T} of $T$ and the definition of $\bm{x}_{k+1}$ (Step 4 of Algorithm \ref{algo:1}), 
\begin{align*}
X_k 
\leq
\mathbb{E}_{\bm{\xi}_k} \left[ \| \bm{x}_{k} - \bm{x}_{k+1} \| \big| \bm{\xi}_{[k-1]} \right]
= \alpha_k \mathbb{E}_{\bm{\xi}_k} \left[ \| \bm{x}_{k} - T_{\bm{\xi}_k} (\bm{x}_{k}) \| \big| \bm{\xi}_{[k-1]} \right].
\end{align*}
Hence, we have 
\begin{align*}
&\mathbb{E}_{\bm{\xi}_k} \left[ \| \bm{x}_{k+1} - T (\bm{x}_{k+1}) \| \big| \bm{\xi}_{[k-1]} \right]\\
&\leq
(1-\alpha_k) \underbrace{\mathbb{E}_{\bm{\xi}_k} \left[ \| \bm{x}_{k} - T (\bm{x}_{k+1}) \| \big| \bm{\xi}_{[k-1]} \right]}_{Y_k}
+ \alpha_k \sqrt{\mathbb{V}_{\bm{\xi}_k} \left[ T_{\bm{\xi}_k} (\bm{x}_{k}) \big| \bm{\xi}_{[k-1]} \right]}\\
&\quad 
+ \alpha_k \mathbb{E}_{\bm{\xi}_k} \left[ \| \bm{x}_{k} - \bm{x}_{k+1} \| \big| \bm{\xi}_{[k-1]} \right].
\end{align*}
Moreover, the triangle inequality implies that
\begin{align*}
Y_k 
&\leq
\mathbb{E}_{\bm{\xi}_k} \left[ \| \bm{x}_{k} -  \bm{x}_{k+1} \| \big| \bm{\xi}_{[k-1]} \right]
+
\mathbb{E}_{\bm{\xi}_k} \left[ \| \bm{x}_{k+1} -  T (\bm{x}_{k+1}) \| \big| \bm{\xi}_{[k-1]} \right].
\end{align*}
Accordingly, we have 
\begin{align*}
&\mathbb{E}_{\bm{\xi}_k} \left[ \| \bm{x}_{k+1} - T (\bm{x}_{k+1}) \| \big| \bm{\xi}_{[k-1]} \right]\\
&\leq
(1-\alpha_k) 
\left\{
\mathbb{E}_{\bm{\xi}_k} \left[ \| \bm{x}_{k} - \bm{x}_{k+1} \| \big| \bm{\xi}_{[k-1]} \right]
+
\mathbb{E}_{\bm{\xi}_k} \left[ \| \bm{x}_{k+1} -  T (\bm{x}_{k+1}) \| \big| \bm{\xi}_{[k-1]} \right]
\right\}\\
&\quad + \alpha_k \sqrt{\mathbb{V}_{\bm{\xi}_k} \left[ T_{\bm{\xi}_k} (\bm{x}_{k}) \big| \bm{\xi}_{[k-1]} \right]} 
+ \alpha_k \mathbb{E}_{\bm{\xi}_k} \left[ \| \bm{x}_{k} - \bm{x}_{k+1} \| \big| \bm{\xi}_{[k-1]} \right]\\
&= 
(1 - \alpha_k)\mathbb{E}_{\bm{\xi}_k} \left[ \| \bm{x}_{k+1} -  T (\bm{x}_{k+1}) \| \big| \bm{\xi}_{[k-1]} \right]
+ \alpha_k \sqrt{\mathbb{V}_{\bm{\xi}_k} \left[ T_{\bm{\xi}_k} (\bm{x}_{k}) \big| \bm{\xi}_{[k-1]} \right]}\\
&\quad 
+ \alpha_k \mathbb{E}_{\bm{\xi}_k} \left[ \| \bm{x}_{k} - T_{\bm{\xi}_k} (\bm{x}_{k}) \| \big| \bm{\xi}_{[k-1]} \right],
\end{align*}
which implies
\begin{align*}
&\mathbb{E}_{\bm{\xi}_k} \left[ \| \bm{x}_{k+1} - T (\bm{x}_{k+1}) \| \big| \bm{\xi}_{[k-1]} \right]\\
&\leq 
\sqrt{\mathbb{V}_{\bm{\xi}_k} \left[ T_{\bm{\xi}_k} (\bm{x}_{k}) \big| \bm{\xi}_{[k-1]} \right]}
+ 
\underbrace{\mathbb{E}_{\bm{\xi}_k} \left[ \| \bm{x}_{k} - T_{\bm{\xi}_k} (\bm{x}_{k}) \| \big| \bm{\xi}_{[k-1]} \right]}_{Z_k}.
\end{align*} 
The triangle inequality and \eqref{jensen} thus ensure that 
\begin{align*}
Z_k 
&\leq
\mathbb{E}_{\bm{\xi}_k} \left[ \| \bm{x}_{k} - T (\bm{x}_{k}) \| \big| \bm{\xi}_{[k-1]} \right]
+
\mathbb{E}_{\bm{\xi}_k} \left[ \| T (\bm{x}_{k}) - T_{\bm{\xi}_k} (\bm{x}_{k}) \| \big| \bm{\xi}_{[k-1]} \right]\\
&=
\| \bm{x}_{k} - T (\bm{x}_{k}) \|
+ 
\sqrt{\mathbb{V}_{\bm{\xi}_k} \left[ T_{\bm{\xi}_k} (\bm{x}_{k}) \big| \bm{\xi}_{[k-1]} \right]}. 
\end{align*}
Therefore, from Proposition \ref{prop:1}(ii), 
\begin{align*}
\mathbb{E}_{\bm{\xi}_k} \left[ \| \bm{x}_{k+1} - T (\bm{x}_{k+1}) \| \big| \bm{\xi}_{[k-1]} \right]
&\leq
\| \bm{x}_{k} - T (\bm{x}_{k}) \|
+ 
2 \sqrt{\mathbb{V}_{\bm{\xi}_k} \left[ T_{\bm{\xi}_k} (\bm{x}_{k}) \big| \bm{\xi}_{[k-1]} \right]}\\
&\leq
\| \bm{x}_{k} - T (\bm{x}_{k}) \|
+ 
\frac{2 \sigma}{\sqrt{b_k}}. 
\end{align*}
The super martingale convergence theorem \cite[Proposition 8.2.10]{bert} with $\sum_{k=0}^{+ \infty} 1/\sqrt{b_k} < + \infty$ ensures the existence of the limit of $\|\bm{x}_k - T (\bm{x}_k) \|$ with probability $1$. This completes the proof.
\qed
\end{proof}

Lemmas \ref{lem:1} and \ref{lem:2} lead to the following.

\begin{theorem}\label{thm:1}
Under Assumption \ref{assum:1}, the sequence $(\bm{x}_k)$ generated by Algorithm \ref{algo:1} with $(\alpha_k)$ and $(b_k)$ satisfying 
\begin{align}\label{condition_1}
\sum_{k=0}^{+\infty} \alpha_k (1 - \alpha_k) = + \infty \text{ and }
\sum_{k=0}^{+ \infty} \frac{1}{\sqrt{b_k}} < + \infty
\end{align}
converges almost surely to a fixed point of $T$.
\end{theorem}

\begin{proof}
Let $\bm{x}^\star \in \mathrm{Fix}(T)$. From $\sum_{k=0}^{+ \infty} \alpha_k/b_k \leq \sum_{k=0}^{+ \infty} 1/\sqrt{b_k} < + \infty$, Lemma \ref{lem:1} (the almost-sure existence of $\lim_{k \to + \infty} \|\bm{x}_k - \bm{x}^\star\|$) provides the almost-sure boundedness of $(\bm{x}_k)$. Hence, there exists a subsequence $(\bm{x}_{k_i})$ of $(\bm{x}_k)$ such that $(\bm{x}_{k_i})$ converges almost surely. Meanwhile, Lemmas \ref{lem:1} and \ref{lem:2} imply that 
\begin{align*}
\lim_{k \to + \infty} \|\bm{x}_k - T (\bm{x}_k)\| = 0 \text{ a.s.}. 
\end{align*}
Accordingly, by letting $\bar{\bm{x}}$ be a convergent random variable of $(\bm{x}_{k_i})$, the continuity of $T$ ensures that $\bar{\bm{x}} \in \mathrm{Fix}(T)$ a.s.. In addition, Lemma \ref{lem:1} leads to the almost-sure existence of $\lim_{k \to + \infty} \|\bm{x}_k - \bar{\bm{x}}\|$. We also have another subsequence $(\bm{x}_{k_j})$ of $(\bm{x}_k)$ such that $(\bm{x}_{k_j})$ converges almost surely to a random variable $\hat{\bm{x}}$. A similar argument to the one above for obtaining $\bar{\bm{x}} \in \mathrm{Fix}(T)$ and the almost-sure existence of $\lim_{k \to + \infty} \|\bm{x}_k - \bar{\bm{x}}\|$ implies that $\hat{\bm{x}} \in \mathrm{Fix}(T)$ and the almost-sure existence of $\lim_{k \to + \infty} \|\bm{x}_k - \hat{\bm{x}}\|$. Therefore, we have 
\begin{align*}
&\lim_{k \to + \infty} \| \bm{x}_k - \bar{\bm{x}} \| = \lim_{i \to + \infty} \|\bm{x}_{k_i} - \bar{\bm{x}}\| = 0 \text{ a.s.},\\ 
&\lim_{k \to + \infty} \| \bm{x}_k - \hat{\bm{x}} \| = \lim_{j \to + \infty} \|\bm{x}_{k_j} - \hat{\bm{x}}\| = 0 \text{ a.s.}.
\end{align*}
Since the triangle inequality ensures that $\| \bar{\bm{x}} - \hat{\bm{x}} \| \leq \| \bar{\bm{x}} - \bm{x}_k \| + \| \bm{x}_k - \hat{\bm{x}} \|$, we have that $\bar{\bm{x}} = \hat{\bm{x}}$ a.s.. That is, $(\bm{x}_k)$ converges almost surely to $\bar{\bm{x}} \in \mathrm{Fix}(T)$, which completes the proof. 
\qed
\end{proof}

Let us examine some examples that satisfy \eqref{condition_1}. A practically used step size is 
\begin{align}\label{constant_1}
\text{[Constant step size] }
\alpha_k = \alpha \in (0,1),
\end{align}
which satisfies $\sum_{k=0}^{+ \infty} \alpha_k (1 - \alpha_k) = + \infty$ from
\begin{align}\label{constant}
\sum_{k=0}^{K-1} \alpha_k (1 - \alpha_k) 
= \sum_{k=0}^{K-1} \alpha (1 - \alpha) 
= \alpha (1 - \alpha) K.
\end{align}
Another step size is 
\begin{align}\label{diminishing_1}
\text{[Diminishing step size] }
\alpha_k = \frac{1}{(k+1)^a} \in (0,1],
\end{align}
where $a \in (0,1]$, which satisfies $\sum_{k=0}^{+ \infty} \alpha_k (1 - \alpha_k) = + \infty$, from
\begin{align}\label{diminishing}
\sum_{k=0}^{K-1} \alpha_k (1 - \alpha_k)
\geq 
\begin{dcases}
\frac{(K+1)^{1-a} -1}{1-a} - \frac{K^{1-2a}}{1-2a} 
&\text{ } \left(a \in \left(0, \frac{1}{2} \right) \right)\\
2 \sqrt{K+1} - \log K - 3
&\text{ } \left(a = \frac{1}{2} \right)\\
\frac{(K+1)^{1-a} -1}{1 - a} - \frac{2a}{2a -1} 
&\text{ } \left(a \in \left(\frac{1}{2}, 1 \right) \right)\\
\log (K+1) - 2  
&\text{ } \left(a =1 \right).
\end{dcases}
\end{align}
An increasing batch size $b_k$ such as 
\begin{align}\label{increasing}
\begin{split}
&\text{[Polynomial increasing batch size] }
b_k = (a k + b_0)^c, \text{ or}\\
&\text{[Exponential increasing batch size] }
b_k = b_0 \delta^k,
\end{split}
\end{align} 
where $a > 0$, $c > 1$, and $\delta > 1$, satisfies $\sum_{k=0}^{+ \infty} 1/\sqrt{b_k} < + \infty$, from
\begin{align}\label{B_value}
\sum_{k=0}^{K-1} \frac{1}{\sqrt{b_k}}
\leq
B 
\coloneqq
\begin{dcases}
\frac{2c-1}{(c-1) \min \{ a, b_0 \}} &\text{ (Polynomial increasing batch size)}\\
\frac{\delta}{(\delta-1) b_0} &\text{ (Exponential increasing batch size)}. 
\end{dcases}
\end{align}
However, a constant batch size $b_k = b$ does not satisfy $\sum_{k=0}^{+ \infty} 1/\sqrt{b_k} < + \infty$. Hence, Theorem \ref{thm:1} says that the use of increasing batch sizes is essential for guaranteeing the convergence of Algorithm \ref{algo:1}. This claim also appears in the convergence analysis of mini-batch SGD \cite{umeda2025increasing}.

\subsection{Convergence rate}
\label{sec:3.2}
Lemma \ref{lem:1} is used here to prove the following theorem on the rate of convergence of Algorithm \ref{algo:1}.

\begin{theorem}\label{thm:2}
Under Assumption \ref{assum:1} and \eqref{condition_1}, the sequence $(\bm{x}_k)$ generated by Algorithm \ref{algo:1} satisfies that, for all $K \in \mathbb{N}$,
\begin{align*}
\min_{k \in [0:K-1]} \mathbb{E} \left[ \| \bm{x}_k - T (\bm{x}_{k})\| \right]
&= 
O \left( \frac{1}{\sqrt{\sum_{k=0}^{K-1} \alpha_k (1 - \alpha_k)}} \right)\\
&\leq
\frac{1}{\sqrt{\sum_{k=0}^{K-1} \alpha_k (1 - \alpha_k)}}
\sqrt{ \| \bm{x}_{0} - \bm{x}^\star \|^2 
+ 
\sigma^2 \sum_{k=0}^{+ \infty} \frac{\alpha_k}{b_k}
}.
\end{align*}
\end{theorem}

\begin{proof}
Let $\bm{x}^\star \in \mathrm{Fix}(T)$. Lemma \ref{lem:1} ensures that, for all $k \in \{0\} \cup \mathbb{N}$,
\begin{align*}
\mathbb{E} \left[ \| \bm{x}_{k+1} - \bm{x}^\star \|^2 \right]
\leq 
\mathbb{E} \left[ \| \bm{x}_{k} - \bm{x}^\star \|^2 \right] + \frac{\sigma^2 \alpha_k}{b_k}
- \alpha_k (1 - \alpha_k) \mathbb{E} \left[ \| \bm{x}_k - T (\bm{x}_{k})\|^2 \right].
\end{align*}
Let $K \in \mathbb{N}$. Summing the above inequality from $k = 0$ to $k = K-1$ leads to the finding that 
\begin{align*}
&
\min_{k \in [0:K-1]} \mathbb{E} \left[ \| \bm{x}_k - T (\bm{x}_{k})\|^2 \right]
\sum_{k=0}^{K-1} \alpha_k (1 - \alpha_k)
\leq
\sum_{k=0}^{K-1} \alpha_k (1 - \alpha_k) \mathbb{E} \left[ \| \bm{x}_k - T (\bm{x}_{k})\|^2 \right]\\
&\leq
\sum_{k=0}^{K-1} \left\{ \mathbb{E} \left[ \| \bm{x}_{k} - \bm{x}^\star \|^2 \right]
- 
\mathbb{E} \left[ \| \bm{x}_{k+1} - \bm{x}^\star \|^2 \right] 
\right\}
+ 
\sigma^2 \sum_{k=0}^{K-1} \frac{\alpha_k}{b_k}\\
&= 
\mathbb{E} \left[ \| \bm{x}_{0} - \bm{x}^\star \|^2 \right]
- 
\mathbb{E} \left[ \| \bm{x}_{K} - \bm{x}^\star \|^2 \right]
+ 
\sigma^2 \sum_{k=0}^{K-1} \frac{\alpha_k}{b_k}\\
&\leq 
\| \bm{x}_{0} - \bm{x}^\star \|^2 
+ 
\sigma^2 \sum_{k=0}^{K-1} \frac{\alpha_k}{b_k}, 
\end{align*} 
which implies that
\begin{align*}
\min_{k \in [0:K-1]} \mathbb{E} \left[ \| \bm{x}_k - T (\bm{x}_{k})\|^2 \right]
\leq
\frac{1}{\sum_{k=0}^{K-1} \alpha_k (1 - \alpha_k)}
\left( \| \bm{x}_{0} - \bm{x}^\star \|^2 
+ 
\sigma^2 \sum_{k=0}^{K-1} \frac{\alpha_k}{b_k}
\right).
\end{align*}
The assertion of Theorem \ref{thm:2} follows from Jensen’s inequality, which ensures that $(\mathbb{E} [ \| \bm{x}_k - T (\bm{x}_{k})\|])^2 \leq \mathbb{E} [ \| \bm{x}_k - T (\bm{x}_{k})\|^2 ]$.
\qed 
\end{proof}

\subsubsection{Concrete convergence rates}
\label{sec:3.2.1}
Theorem \ref{thm:2} with \eqref{constant} and \eqref{B_value} indicates that Algorithm \ref{algo:1} using a constant step size defined by \eqref{constant_1} and an increasing batch size defined by \eqref{increasing} has the following convergence rate:
\begin{align}\label{rate}
\begin{split}
\min_{k \in [0:K-1]} \mathbb{E} \left[ \| \bm{x}_k - T (\bm{x}_{k})\| \right]
= O \left(\frac{1}{\sqrt{K}} \right)
\leq
\sqrt{
\frac{\| \bm{x}_{0} - \bm{x}^\star \|^2 
+ 
\sigma^2 \alpha B}{\alpha (1 - \alpha) K}
},
\end{split}
\end{align}
where $B$ is defined as in \eqref{B_value}. Meanwhile, Theorem \ref{thm:2} with \eqref{diminishing} and \eqref{B_value} indicates that Algorithm \ref{algo:1} using a diminishing step size defined by \eqref{diminishing_1} and an increasing batch size defined by \eqref{increasing} has the following convergence rate:
\begin{align*}
\min_{k \in [0:K-1]} \mathbb{E} \left[ \| \bm{x}_k - T (\bm{x}_{k})\| \right]
= 
\begin{dcases}
O \left( \frac{1}{\sqrt{K^{1-2a}}} \right) 
&\text{ } \left(a \in \left(0, \frac{1}{2} \right) \right)\\
O \left( \frac{1}{\sqrt{\sqrt{K} - \log K}} \right)
&\text{ } \left(a = \frac{1}{2} \right)\\
O \left( \frac{1}{\sqrt{K^{1-a}}} \right)
&\text{ } \left(a \in \left(\frac{1}{2}, 1 \right) \right)\\
O \left( \frac{1}{\sqrt{\log K}} \right)  
&\text{ } \left(a =1 \right).
\end{dcases}
\end{align*}
Hence, Theorem \ref{thm:2} with \eqref{rate} says that constant step sizes are essential for guaranteeing a fast rate $O(1/\sqrt{K})$ of convergence of Algorithm \ref{algo:1}.

Theorem \ref{thm:2} further ensures that, even if a constant batch size $b$ is used, we have 
\begin{align*}
\min_{k \in [0:K-1]} \mathbb{E} \left[ \| \bm{x}_k - T (\bm{x}_{k})\|^2 \right]
\leq
\frac{1}{\sum_{k=0}^{K-1} \alpha_k (1 - \alpha_k)}
\left( \| \bm{x}_{0} - \bm{x}^\star \|^2 
+ 
\frac{\sigma^2}{b} \sum_{k=0}^{K-1} \alpha_k
\right).
\end{align*}
However, from $1 \leq \sum_{k=0}^{K-1} \alpha_k / \sum_{k=0}^{K-1} \alpha_k (1 - \alpha_k)$, the use of constant batch sizes does not guarantee convergence of $\min_{k \in [0:K-1]} \mathbb{E} [ \| \bm{x}_k - T (\bm{x}_{k})\|]$ to $0$, as is evident as well from Theorem \ref{thm:1}.

\subsection{Applying Theorems \ref{thm:1} and \ref{thm:2} to Example \ref{exp:1}} 
Let us apply our results, Theorems \ref{thm:1} and \ref{thm:2} and \eqref{rate}, in the case of a constant step size $\alpha$ defined by \eqref{constant} and an increasing batch size $b_k$ defined by \eqref{increasing} to the problems in Example \ref{exp:1}.

\textit{Example \ref{exp:1}}(i) [Convex feasibility problem] The sequence $(\bm{x}_k)$ generated for all $k \in \{0\} \cup \mathbb{N}$ by 
\begin{align*}
\bm{x}_{k+1} = (1 - \alpha) \bm{x}_k + \frac{\alpha}{b_k} \sum_{i=1}^{b_k} P_{\xi_{k,i}} (\bm{x}_k) 
\end{align*}
converges almost surely to a point in $C = \bigcap_{i=1}^n C_i$ (i.e., a fixed point of $T$ defined by \eqref{cfp_t}) with the following convergence rate:
\begin{align*}
\min_{k \in [0:K-1]} \mathbb{E} \left[  \left\| \frac{1}{n} \sum_{i=1}^n \left( \bm{x}_k - P_i (\bm{x}_k) \right) \right\| \right] 
&= O \left( \frac{1}{\sqrt{K}} \right)\\
&\leq
\sqrt{ 
\frac{\| \bm{x}_{0} - \bm{x}^\star \|^2 
+ 
(r + \|\bm{x}^\star\|)^2 \alpha B}{\alpha (1 - \alpha) K}}.
\end{align*} 

\textit{Example \ref{exp:1}}(ii) [Zero point problem] The sequence $(\bm{x}_k)$ generated for all $k \in \{0\} \cup \mathbb{N}$ by 
\begin{align*}
\bm{x}_{k+1} = \bm{x}_k - \frac{\alpha \beta}{b_k} \sum_{i=1}^{b_k} A_{\xi_{k,i}} (\bm{x}_k) 
\end{align*}
converges almost surely to a point in $\bigcap_{i=1}^n A_i^{-1}(\bm{0})$ (i.e., a fixed point of $T$ defined by \eqref{zero_t}) with the following convergence rate:
\begin{align*}
\min_{k \in [0:K-1]} \mathbb{E} \left[  \left\| \frac{1}{n} \sum_{i=1}^n A_i (\bm{x}_k)  \right\| \right] 
&= O \left( \frac{1}{\sqrt{K}} \right)\\
&\leq
\sqrt{ 
\frac{\| \bm{x}_{0} - \bm{x}^\star \|^2 
+ 
r^2 \beta^{-1} (2 \gamma - \beta)^{-1} \alpha B}{\alpha (1 - \alpha) \beta K}}.
\end{align*}

\textit{Example \ref{exp:1}}(iii) [Convex minimization problem] The sequence $(\bm{x}_k)$ generated for all $k \in \{0\} \cup \mathbb{N}$ by 
\begin{align*}
\bm{x}_{k+1} = \bm{x}_k - \frac{\alpha \eta}{b_k} \sum_{i=1}^{b_k} \nabla f_{\xi_{k,i}} (\bm{x}_k) 
\end{align*}
converges almost surely to a global minimizer of $f \coloneqq \frac{1}{n} \sum_{i=1}^n f_i$ (i.e., a fixed point of $T$ defined by \eqref{convex_t}) with the following convergence rate:
\begin{align*}
\min_{k \in [0:K-1]} \mathbb{E} \left[  \left\| \nabla f (\bm{x}_k)  \right\| \right] 
= O \left( \frac{1}{\sqrt{K}} \right)
\leq
\sqrt{ 
\frac{\| \bm{x}_{0} - \bm{x}^\star \|^2 
+ 
\eta^2 \sigma_g^2 \alpha B}{\alpha (1 - \alpha) \eta K}}.
\end{align*}
This result coincides with \cite[Theorem 3.2]{umeda2025increasing} showing that mini-batch SGD has an $O(1/\sqrt{K})$ convergence rate. 

\section{Conclusion and Future Work}
This paper considered a stochastic fixed point problem for nonexpansive mappings and presented a convergence analysis of the mini-batch stochastic Krasnosel'ski\u\i-Mann algorithm for solving it. The analysis showed that the algorithm using an increasing batch size converges almost surely to a fixed point of the expectation of stochastic nonexpansive mappings. This paper also presented a convergence rate analysis demonstrating that the use of a constant step size leads to faster convergence of the algorithm compared with the use of diminishing step sizes.

The Halpern algorithm is a useful fixed point algorithm for nonexpansive mappings. Hence, in the future, we should verify whether this algorithm, defined by 
\begin{align*}
\bm{x}_{k+1} = \alpha_k \bm{x}_0 + (1 - \alpha_k) T_{\bm{\xi}_k} (\bm{x}_k),
\end{align*}
can be applied to the stochastic fixed point problem. While this paper assumed the boundedness of the variance of a stochastic mapping (Assumption \ref{assum:1}(A2)), the previously reported results in \cite{pmlr-v97-simsekli19a,pmlr-v139-garg21b,pmlr-v238-battash24a,ahn2024linear} showed that stochastic noise in practical machine learning may exhibit heavy-tailed behavior, violating the bounded-variance assumption. Hence, we should also verify whether the mini-batch stochastic Krasnosel'ski\u\i-Mann algorithm converges under heavy-tailed noise.

\section*{Compliance with Ethical Standards}
The author declares that there is no conflict of interest. This article does not contain any studies with human participants or animals performed by the author.

\begin{acknowledgements}
This work was supported by JSPS KAKENHI Grant Number 24K14846. The author declares no conflict of interest.
\end{acknowledgements}

\bibliographystyle{spmpsci}
\bibliography{bib}
\end{document}